\documentclass[12pt]{article}
\usepackage{amsmath,amssymb,amsfonts,amsthm}
\newcounter{item}[section]
\newcounter{kirshr}
\newcounter{kirsha}
\newcounter{kirshb}
\newenvironment{enumroman}{\setcounter{kirshr}{1}
\begin{list}{(\roman{kirshr})}{\usecounter{kirshr}} }{\end{list}}
\newenvironment{enumarab}{\setcounter{kirshb}{1}
\begin{list}{(\arabic{kirshb})}{\usecounter{kirshb}} }{\end{list}}
\newtheorem{theorem}{Theorem}[section]

\newtheorem{lemma}[theorem]{Lemma}
\newtheorem{corollary}[theorem]{Corollary}
\newenvironment{demo}[1]{\noindent{\bf #1.}\upshape\mdseries}
{\nopagebreak{\hfill\rule{2mm}{2mm}\nopagebreak}\par\normalfont}
\theoremstyle{definition}

\newtheorem{definition}[theorem]{Definition}

\def\C{{\mathfrak{C}}}
\def\Fm{{\mathfrak{Fm}}}

\def\Nr{{\mathfrak{Nr}}}
\def\Fr{{\mathfrak{Fr}}}
\def\Sg{{\mathfrak{Sg}}}
\def\Fm{{\mathfrak{Fm}}}
\def\A{{\mathfrak{A}}}
\def\B{{\mathfrak{B}}}
\def\C{{\mathfrak{C}}}
\def\D{{\mathfrak{D}}}
\def\M{{\mathfrak{M}}}
\def\N{{\mathfrak{N}}}

\def\GPEA{{\sf{GPEA}}}

\def\Bl{{\mathfrak{Bl}}}
\def\T{{bf T}}

\def\CA{{\bf CA}}

\def\SA{{\bf SA}}

\def\Lf{{\bf Lf}}
\def\Dc{{\bf Dc}}

\def\RCA{{\bf RCA}}
\def\SA{{\bf SA}}
\def\Rd{{\ Rd}}
\def\(R)RA{{\bf (R)RA}}
\def\RA{{\bf RA}}
\def\RRA{{\bf RRA}}
\def\Dc{{\bf Dc}}

\def\Dc{{\bf Dc}}
\def\Mn{{\bf Mn}}

\def\Mn{{\bf Mn}}
\def\P{{\mathcal P}}

 \def\CA{{\sf CA}}
\def\B{{\sf B}}

\def\Nr{{\mathfrak{Nr}}}

\def\Ra{{\mathfrak{Ra}}}

\def\Ra{{\mathfrak{Ra}}}
\def\Nr{{\mathfrak{Nr}}}

\def\A{{\mathfrak{A}}}
\def\B{{\mathfrak{B}}}
\def\C{{\mathfrak{C}}}
\def\D{{\mathfrak{D}}}

\def\A{{\mathfrak{A}}}
\def\B{{\mathfrak{B}}}
\def\C{{\mathfrak{C}}}
\def\D{{\mathfrak{D}}}

\def\Ig{{\mathfrak{Ig}}}

\def\L{{\mathfrak{L}}}
\def\Fr{{\mathfrak{Fr}}}

\def\L{{\mathfrak{L}}}
\def\CA{{\bf CA}}
\def\RA{{\bf RA}}
\def\RRA{{\bf RRA}}
\def\RCA{{\bf RCA}}

\def\Rd{{\mathfrak{Rd}}}
\def\Cs{{\sf{Cs}}}
\def\QRA{{\sf{QRA}}}

\begin{document}


\title{On various amalgamation bases for cylindric algebras, and relation structures}

\author{Tarek Sayed Ahmed\thanks{Email:
rutahmed@gmail.com}\\Department of Mathematics,\\  Faculty of Science,\\
Cairo University,\\ Giza, Egypt.}
%

 \maketitle

\begin{abstract} We characterize completely the amalgamation, strong amalgamation, and superamalgamation base
of several classes of representable algebras. The characerization is via special neat embeddings. We also study expansions of such algebras that have the 
superamalgamation property.
\footnote{ 2000 {\it Mathematics Subject Classification.} Primary 03G15. Secondary 03C05, 03C40

keywords: Algebraic logic, cylindric algebras, amalgamation base}


\end{abstract}

\section{Introduction}

Pigozzi and Comer prove that for $\alpha>1$, the class of representable cylindric algebras of dimension $\alpha$ 
fails to have the amalgamation property \cite{P}. The same result was established for many 
other algebras that are cousins to cylindric algebras, namely,
quasi-polyadic algebras, quasi-polyadic equality algebras and Pinter's substitution algebras.
For finite dimensions $n\geq 2$, this was done by Comer, and for infinite dimensions this was done by the present author.

In the book `Algebraic Logic', \cite{AMN}, edited by Andr\'eka, Monk and N\'emeti, several questions
were collected by the Editors at the final chapter. In this note we answer the following questions:
\begin{enumarab}

\item Characterize the amalgamation base of $\RCA_{\alpha}$ ($\A$ is in the amalgamation base of $K$, $APbase(K)$,  
if any two algebras in $K$ having $\A$ as a common subalgebra can be amalgamated over $\A$ in $K$).

\item Is $APbase(\RCA_{\alpha})\subseteq APbase(\CA_{\alpha})$?

\end{enumarab}
In what follows we answer these questions (and one more) only for cylindric algebras. 
The proof is exactly the same for other algebras by the results in \cite{Stud}, namely, theorem 12, and corollary 13, 
which address the more general notion of 
systems of varieties definable by schemes.

The two questions are sound; the first motivated by the negative results of Comer and Pigozzi, while the second is motivated by the fact that 
the for $\alpha\geq 2$, the inclusion $\RCA_{\alpha}\subset \CA_{\alpha}$ is proper, and, in fact, for $\alpha\geq 3$ 
this gap cannot be axiomatized by any reasonable finite schema.

Some  comments, on the nature of the two questions,  are in order:

\begin{enumarab}
\item The word characterize here is vague; it could mean, for example, finding a set of first order formulas that defines the amalgamation base
implicitly meaning that the base is an elementary class. It could also mean to
find out how this class behaves with respect to the algebraic operations $H$ (taking homomorphic images), $S$ (forming subalgebras) and 
$P$ (taking products); are we lucky enough that it turns out to be 
an  equational class? These quesions are motivated by the fact that now we have the amalgamation base infront of us, 
so the most pressing need is to try to classify it.
Classifying is a kind of defining by a set of hopefully first order axioms.
In case it is elementary, or even better equational, is it perhaps finitely axiomatizable 
over the class of representable algebras? These are all fair questions.

However, we chose a different approach in attacking the first problem, that  we believe is founded on good reasons. 
We characterize this class by imposing conditions on how such algebras {\it neatly embed} into other algebras
$\omega$ extra dimensions; the point is that they always do (by the neat embedding theorem of Henkin).
But because the class of representable algebras fails the amalgamation property, so there has to be an additional 
condition on how they embed, that characterizes that they {\it are} in the amagamation base. 

We find a condition that is quite natural, and satisfactory, in so far as the neat embedding theorem of Henkin of characerizing 
the class of representable
algebras is satisfactory; the hitherto exhibited  condition  completely characterizes this class.

Such a  condition was actually found in \cite{AU}; it is reported in \cite{book}, 
but in the former reference only one implication is proved; here we
prove the equivalence, answering an implicit  question in both papers, and a question in \cite{AMN}.

\item  The question so formulated lends itself to further questions along the same line; indeed it allows a very natural 
generalization; by characterizing those algebras
that render stronger amalgamation properties.

In this connection, we characterize (via neat embeddings as well) 
the superamalgamation base and the strong amalgamation base of the class of representable algebras. 
We show that they are the same, which is slighly surprising and apparently in conflict
with the deep Sagi-Shelah result; there are varieties of finite dimensional representable algebras that have 
the strong amalgamation property but does not have the superamalgamation property \cite{Shelah}.

\item We also show that the superamalgamation base of $RCA_{\omega}$ is not contained in $APbase(CA_{\omega})$. 
This means that there is an algebra that superamalgamates any two algebras of which it is the common subalgebra, 
if these algebras are representable,  but there are non-representable algebras having this algebra as a common subalgebra, and they  cannot  be only 
amalgamated. This gives a very strong negative result to (2). This result is mention in \cite{Mad}, and a sketch of proof is given in \cite{Sayed}. 

\item The conditions we formulate are also given (syntactically) 
in the very general context of systems of varieties definable by schemes in \cite{Stud}.
In op.cit , the equivalence was shown to hold modulo a certain (semantical) condition, 
which cannot be dealt with in systems of varieties.  The main novelty here is that we prove this condition, hence the desired 
equivalence.
\end{enumarab}

Ordinals considered, unless otherwise specified, are infinite. Some of our results though apply to the finite dimensional case.
These occasions will be explicity mentioned. 

\section{Ctlindric algebras}

We follow the notation and terminology of the monograph \cite{HMT1} and \cite{AFN}. 
In particular $\CA_{\beta}$ stands for the class of cylindric algebras of dimension $\beta$.
$\Cs_{\beta}$ stands for the class of cylindric set algebras of dimension $\beta$. 
The class of representable algebras of dimension $\beta$, which coincides with the class of subdirect products of set algebras,
For $x\in \A$, let $\Delta x=\{i\in \beta: {\sf c}_ix\neq x\}$. 

For an algebra $\B,$ $Nr_{\alpha}\B=\{x\in B: \Delta x\subseteq \alpha\}$ is a subuniverse of $\Rd_{\alpha}\B$.
Then the algebra $\Nr_{\alpha}\B\in \CA_{\alpha}$ with universe $Nr_{\alpha}\B$ is called the neat $\alpha$ reduct of $\B$.

For $K\subseteq \CA_{\beta}$ and $\alpha<\beta$,
$\Nr_{\alpha}K=\{\Nr_{\alpha}\B: \B\in K\}$.

The neat embedding theorem says that  for any $\alpha$ and any $\A\in CA_{\alpha}$,
$\A$ is representable iff there exists a $\B\in CA_{\alpha+\omega}$ and an injective homomorphism $:\A\to \Nr_\B$ 
called a neat embedding. The notion of neat reducts and neat embeddng is lengthly discussed in \cite{Sayed}.
We recall the conditions from \cite{AU} and \cite{stud}; for an intuitive explanation for these seemingly complicated conditions, the reader is 
referred to either reference:

\begin{definition} Let $\A\in \RCA_{\alpha}$. Then $\A$ has the $UNEP$ (short for unique neat embedding property) 
if for all $\A'\in \CA_{\alpha}$, $\B$, $\B'\in \CA_{\alpha+\omega},$
isomorphism $i:\A\to \A'$, embeddings  $e_A:\A\to \Nr_{\alpha}\B$ and $e_{A'}:\A'\to \Nr_{\alpha}\B'$ such that $\Sg^{\B}e_A(A)=\B$ and $\Sg^{\B'}e_{A'}(A)'=\B'$, there exists
an isomorphism $\bar{i}:\B\to \B'$ such that $\bar{i}\circ e_A=e_{A'}\circ i$. 
\end{definition}

\begin{definition} Let $\A\in \RCA_{\alpha}$. Then $\A$ has the $NS$ property (short for neat reducts commuting with forming subalgebras) 
if for all $\B\in \CA_{\alpha+\omega}$ if $\A\subseteq \Nr_{\alpha}\B$ then for all 
$X\subseteq A,$ $\Sg^{\A}X=\Nr_{\alpha}\Sg^{\B}X$.
\end{definition}
We recall the definition of the amalgamation base and the super amalgamation base:

\begin{definition} Let $L$ be a class of algebras. 
\begin{enumroman}
\item  $\A_0\in L$ is in the amalgamation base of $L$ if for all $\A_1, \A_2\in L$ and monomorphisms 
$i_1:\A_0\to \A_1$ $i_2:\A_0\to \A_2$ 
there exist $\D\in L$
and monomorphisms $m_1:\A_1\to \D$ and $m_2:\A_2\to \D$ such that $m_1\circ i_1=m_2\circ i_2$. 
\item Let everything be as in (i) and assume that the algebras considered have a Boolean reduct. If in addition, $(\forall x\in A_j)(\forall y\in A_k)
(m_j(x)\leq m_k(y)\implies (\exists z\in A_0)(x\leq i_j(z)\land i_k(z) \leq y))$
where $\{j,k\}=\{1,2\}$, then we say that $\A_0$ lies in the super amalgamation base of $L$. 
\end{enumroman}
\end{definition}
As pointed out before, we will characterize the amalgamation base of $\RCA_{\alpha}$ using neat embeddings. 
For this we need the next two crucial 
lemmas, on which the proof hinges. These two lemmas are  missing from \cite{AU} and \cite{Stud}.

\begin{lemma} Let $\A, \B$ be algebras having the same similarity type. Then the following two conditions are equivalent:
\begin{enumroman}
\item $h:\A\to \B$ is a homomorphism
\item $h$ is a subalgeba of $\A\times \B$ and $h$ is a function with $Domh=A$
\end{enumroman}
\end{lemma}
\begin{demo}{Proof} \cite{HMT1} 0.3.47
\end{demo}
\begin{lemma}
Let $K=\{\A\in \CA_{\alpha+\omega}: \A=\Sg^{\A}\Nr_{\alpha}\A\}$. Let $\A,\B\in K$ and suppose that $f:\Nr_{\alpha}\A\to \Nr_{\alpha}\B$ is an isomorphism.
Then there exists an isomorphism $g:\A\to \B$ such that $f\subseteq g$
\end{lemma}
\begin{demo}{Proof} Let $g=\Sg^{\A\times \B}f$. It suffices to show, by the previous lemma, 
that $g$ is a one to one function with domain $A$.
$$Dom g=Dom\Sg^{\A\times \B}f=\Sg^{\A}Domf=\Sg^{\A}\Nr_{\alpha}\A=\A.$$
By symmetry it is enough to show that $g$ is a function.  We first prove the following (*)
 $$ \text { If } (a,b)\in g\text { and }  {\sf c}_k(a,b)=(a,b)\text { for all } k\in \alpha+\omega\sim \alpha, \text { then } f(a)=b.$$
Indeed,
$$(a,b)\in \Nr_{\alpha}\Sg^{\A\times \B}f=\Sg^{\Nr_{\alpha}(\A\times \B)}f=\Sg^{\Nr_{\alpha}\A\times \Nr_{\alpha}\B}f=f.$$
Here we are using that $\A\times \B\in \Dc_{\alpha+\omega}$, so that  $\Nr_{\alpha}\Sg^{\A\times \B}f=\Sg^{\Nr_{\alpha}(\A\times \B)}f.$
Now suppose that $(x,y), (x,z)\in g$.
Let $k\in \alpha+\omega\sim \alpha.$ Let $\Delta$ denote symmetric difference. Then
$$(0, {\sf c}_k(y\Delta z))=({\sf c}_k0, {\sf c}_k(y\Delta z))={\sf c}_k(0,y\Delta z)={\sf c}_k((x,y)\Delta(x,z))\in g.$$
Also,
$${\sf c}_k(0, {\sf c}_k(y\Delta z))=(0,{\sf c}_k(y\Delta z)).$$ 
Thus by (*) we have  $$f(0)={\sf c}_k(y\Delta z), \text { for any } k\in \alpha+\omega\sim \alpha.$$
Hence ${\sf c}_k(y\Delta z)=0$ and so $y=z$.

\end{demo}.

The next main theorem, follows immediately from the previous lemma, together with theorem 12, p. 33 in \cite{Stud}
\begin{theorem} Let $\alpha$ be an ordinal. Let $\C\in \RCA_{\alpha}$.
\begin{enumroman}
\item $\C$ has $UNEP$ if and only if $\C\in APbase$.
\item $\C$ has $UNEP$ and $NS,$  if and only if $\C\in SUPAPbase$.
\end{enumroman}
\end{theorem}
\begin{demo}{Proof}\cite{Stud}
\end{demo}

The following result is also mentioned in \cite{AFN} with only a sketchy proof.

\begin{theorem} Let $\alpha$ be infinite. Let $\D$ the full set algebra with unit $^{\alpha}\alpha$. Let $\M$ be its minimal subalgeba.
Then $\M\in SUPAPbase(\RCA_{\alpha})\sim APbase(\CA_{\alpha})$
\end{theorem}

\begin{demo}{Proof} There is a non-representable algebra $\C$ whose minimal subalgebra is isomorphic to $\M$, and such that $\C$ and $\D$ 
cannot be embedded in a common algebra. By $\Mn_{\alpha}\subseteq \Dc_{\alpha}\subseteq SUPAPbase(\RCA_{\alpha})$ we are done.
\end{demo}


\begin{theorem} The following conditions are equivalent
\begin{enumarab}
\item $\A$ has $UNEP$ and $NS$
\item $\A\in SAPbase(\RCA_{\alpha})$
\item $\A\in SUPAP(\RCA_{\alpha})$
\end{enumarab}
\end{theorem}
\begin{demo}{Proof} \cite{Stud} corollary 13
\end{demo}

\begin{theorem} 
$APbase(\RCA_{\alpha})\nsubseteq APbase(\RCA_{\alpha})$ and same for the super and strong amalgmation base.
\end{theorem}
\begin{demo}{Proof} Because $SUPAbase(\RCA_{\alpha})\nsubseteq APbase(\CA_{\alpha})$
\end{demo}

\section{Expansions of cylindric algebras}

Let $\alpha$ be an infinite ordinal and $G\subseteq {}^{\alpha}{\alpha}$. Let $\bold T$ denote the semigroup generated by $G$.
$[i|j]$ denotes the replacement that sends $i$ to $j$
and fixes the other elements. Let $[i,j]$ denote the transposition 
that interchanges $i$ and $j$. We assume that $G$ contains all replacements and transpositions.
Let $G^*$ denote the set of all finite words on $G$, i.e. 
$G^*=\bigcup_{n\in \omega}{}^nG$. For $u,w\in G^*$, 
$u^{\cap}w$, or simply $uw$, denotes the
{\it concatenation} of $u$ and $w$. Recall that $(G^*, ^{\cap})$ is the free semigroup
generated by $G$. We let $\hat{}:G^*\to \bold T$ denote the unique 
anti-homomorphism 
extending the identity inclusion $Id:G\to \bold T$. 
Let $w\in G^*$. For $u\in G,$ let ${\sf s}_u(x)$ be a unary term. Suppose that 
$w=\langle w_0,\cdots ,w_{n-1}\rangle$, with $w_i\in G$. 
Then ${\sf s}_{w}(x)$ denotes the unary term 
${\sf s}_{w_{n-1}}(\cdots ({\sf s}_{w_0}(x))\cdots)$.  
$\Sigma_G$ denotes the following set of equations in one variable:
$$\Sigma_G:=\{{\sf s}_{u}(x)={\sf s}_{w}(x): u,w\in G^*\text { and } u\hat {}=w\hat {}\}.$$
In what follows $G$ will denote a set of transformations on $\alpha$ and $\bold T$ will denote the semigroup generated by $G.$

\begin{definition}

\begin{enumroman}

\item By a $G$ polyadic algebra of dimension $\alpha$,  we understand an algebra of the form
$$\A=\langle A,+,.,-,0,1,{\sf c}_i,{\sf s}_{\tau} \rangle_{i\in \alpha ,\tau\in \bold T}$$
where ${\sf c}_i$ ($i\in \alpha$) and ${\sf s}_{\tau}$ ($\tau\in \bold T)$ are unary 
operations on $A$ such that postulates $(P_1-P_{10})$ 
below hold for $\tau,\sigma\in \bold T$
and all $i,j,k\in \alpha$.

$(P_1)$ $\langle A,+,.,-,0,1\rangle$ is a boolean algebra, henceforth denoted
by $Bl\A$.

$(P_2)$ $x\leq {\sf c}_ix={\sf c}_i{\sf c}_ix,\ {\sf c}_i(x+y)={\sf c}_ix+{\sf c}_iy,\ {\sf c}_i(-{\sf c}_ix)=-{\sf c}_ix,\ {\sf c}_i{\sf c}_jx={\sf c}_j{\sf c}_ix$.

In other words ${\sf c}_i$ is an additive closure operator, and ${\sf c}_i,{\sf c}_j$ commute.

$(P_3)$ ${\sf s}_{\tau}$ is a boolean endomorphism.

$(P_4)$$\Sigma_G$. In particular ${\sf s}_{\tau}{\sf s}_{\sigma}x={\sf s}_{\tau\circ \sigma}x$
and ${\sf s}_{Id}x=x$.

$(P_5)$${\sf s}_{\tau}{\sf c}_ix={\sf s}_{\tau[i|j]}{\sf c}_ix$.

$\tau[i|j]$ is the transformation that agrees with $\tau$ on 
$\alpha\smallsetminus\{i\}$ and $\tau[i|j](i)=j$.

$(P_6)$ ${\sf s}_{\tau}{\sf c}_ix={\sf c}_j{\sf s}_{\tau}x$ if $\tau^{-1}(j)=\{i\}$.  

$(P_7)$ ${\sf c}_i{\sf s}_{[i|j]}x={\sf s}_{[i|j]}x$ if $i\neq j$

$(P_8)$ ${\sf s}_{[i|j]}{\sf c}_ix={\sf c}_ix$.

$(P_9)$ ${\sf s}_{[i|j]}{\sf c}_kx={\sf c}_k{\sf s}_{[i|j]}x$ whenever $k\notin \{i,j\}$.

$(P_{10})$ ${\sf c}_i{\sf s}_{[j|i]}x={\sf c}_j{\sf s}_{[i|j]}x$.

\item By a $G$ polyadic equality algebra, a $\GPEA_{\alpha}$ for short, 
we understand an algebra  of the form  

$$\B=\langle B,+,.,-,0,1, {\sf c}_i,{\sf s}_{\tau}, {\sf d}_{ij}\rangle_{i,j\in \omega,\tau\in \bold T}$$
such that 
${\sf d}_{ij}\in A$ for all $i,j\in \omega$ and 
$\langle A,+,.,-,0,1,{\sf c}_i, {\sf s}_{\tau} \rangle_{i\in \omega,\tau\in \bold T}$
is a $G$ algebra such that the following hold for all $k,l\in \omega$
and all $\tau\in \T:$

$(P_{11})$ ${\sf d}_{kk}=1$

$(P_{12})$ ${\sf s}_{\tau}{\sf d}_{kl}={\sf d}_{\tau(k), \tau(l)}.$

$(P_{13})$ $x\cdot {\sf d}_{kl}\leq {\sf s}_{[k|l]}x$

\end{enumroman}

\end{definition}
\begin{theorem} Let $G=\{[i|j], suc, pred\}$ and $V=Mod(1-13)$. Let 
$\A=\Fr_{\omega}V$ 
Let $X_1, X_2$ be subsets of the set of free generators. Let $a\in \Sg^{\Rd_{ca}\A}X_1$
and $b\in \Sg^{\Rd_{ca}\A}X_2$ such that $a\leq b$. Then there exists $c\in \Sg^{\A}(X_1\cap X_2)$ such that
$a\leq c\leq b.$
\end{theorem}
\begin{demo}{Proof} 
Let $\B$ be an $\omega$ dilation.  Assume that there is no interpolant.
As above there exists no interpolant in $\B$. Let $F_1$ and $F_2$ be Henkin ultrafilters as in lemma 1.6.
Let $\alpha=\omega+\omega$. 
Now for all $x\in \Sg^{\B}(X_1\cap X_2)$ we have 
$x\in F_1$ iff $x\in F_2$. 
In particular, for $i,j<\alpha$, we have (++) 
$$\ \ {\sf d}_{ij}\in F_1\text { iff } {\sf d}_{ij}\in F_2.$$
Let $k\in \{1,2\}$.
$$E_k=\{(i,j)\in {}^2{\alpha}: {\sf d}_{ij}\in F_k\}.$$
$E_k$ is an equivalence relation on $\alpha$
$E_1=E_2=E$, say. 
Let $M= \alpha/E$ and for $i\in \omega$, let $q(i)=i/E$. 
Let $W$ be the weak space $^{\omega}M^{(q)}.$
For $h\in W,$ we write $h=\bar{\tau}$ if $\tau\in {}^{\omega}\alpha^{(Id)}$ is such that
$\tau(i)/E=h(i)$ for all $i\in \omega$. $\tau$ of course may
not be unique.
For $\tau\in{} ^{\omega}\alpha^{(Id)},$ let $\tau^+=
\tau\cup (Id\upharpoonright \alpha\setminus \omega.)$
That is $\tau^+$ agrees with $\tau$ on $\omega$ 
and is the identity on $\alpha\setminus \omega$.
For $m\in \{1,2\},$ let $\B_m=\Sg^{\Rd_{QPEA}\A}X_m$.
Define $f_m$ from $\B_m$ to the full weak set algebra with unit $W$ as follows:
$$f_m(x)=\{ \bar{\tau} \in W:  {\sf s}_{\tau^+}^{\A}x\in F_m\}, \text { for } x\in \B_m.$$ 
Then we claim that $f_m$ is a homomorphism from $\B_m$ into 
$\langle \B(W), {\sf C}_i, {\sf P}_{ij}, {\sf D}_{ij}\rangle_{i,j\in \omega}.$ Notice first that $Id \in f_1(a)\cap f_2(-c)$.
For brevity write $F$ for $F_m$, $f$ for $f_m$ and $\B$ for $\B_m$.
First thing to do is to make sure that the map $f$ is well defined. 
For this, it clearly suffices to show that for  
$\sigma, \tau\in {}^{\omega}{\alpha}^{(Id)}$ 
and $x\in A,$ if $\sigma(i)E \tau(i)$ for all $i<\omega,$
then $${\sf s}_{\tau^+}x\in F\text { iff } {\sf s}_{\sigma^+}x\in F.$$
This can be proved by induction on the cardinality of the (finite) set 
$$J=\{i\in \omega: \sigma(i)\neq \tau(i)\}.$$
If $J$ is empty, the result is obvious. 
Otherwise assume that $k\in J$. We introduce a piece of notation.
For $\eta\in V={}^{\alpha}\alpha^{Id}$ let  
$\eta(k\mapsto l)$ for the $\eta'$ that is the same as $\eta$ except
that $\eta'(k)=l.$ 
Now take any 
$$\lambda\in  \{\eta\in \alpha: (\sigma^+){^{-1}}\{\eta\}= (\tau^+){^{-1}}\{\eta\}=
\{\eta\}\}\smallsetminus \Delta x.$$
Recall that $\Delta x=\{i\in \alpha: c_ix\neq x\}$ and that $\alpha\setminus \Delta x$
is infinite. We freely use properties of substitutions for cylindric algebras.
We have by 1.3 (iv) \cite[1.11.11(i)(iv)]{HMT1}
and  1.3 (vi)that (a)
$${\sf s}_{\sigma^+}x={\sf s}_{\sigma k}^{\lambda}{\sf s}_{\sigma^+ (k\mapsto \lambda)}x.$$
Also by 1.3 (iii) (vi ) we have (b)
$${\sf s}_{\tau k}^{\lambda}({\sf d}_{\lambda, \sigma k}. {\sf s}_{\sigma^+} x)
={\sf d}_{\tau k, \sigma k} {\sf s}_{\sigma^+} x,$$
and (c)
$${\sf s}_{\tau k}^{\lambda}({\sf d}_{\lambda, \sigma k}.{\sf s}_{\sigma^+(k\mapsto \lambda)}x)$$
$$= {\sf d}_{\tau k,  \sigma k}.{\sf s}_{\sigma^+(k\mapsto \tau k)}x.$$

By 1.3 (i) we have (d)

$${\sf d}_{\lambda, \sigma k}.{\sf s}_{\sigma k}^{\lambda}{\sf s}_{{\sigma}^+(k\mapsto \lambda)}x=
{\sf d}_{\lambda, \sigma k}.{\sf s}_{{\sigma}^+(k\mapsto \lambda)}x$$

Then by (b), (a), (d) and (c), we get,

$${\sf d}_{\tau k, \sigma k}.{\sf s}_{\sigma}^+ x= 
{\sf s}_{\tau k}^{\lambda}({\sf d}_{\lambda,\sigma k}.{\sf s}_{\sigma}^+x)$$
$$={\sf s}_{\tau k}^{\lambda}({\sf d}_{\lambda, \sigma k}.{\sf s}_{\sigma k}^{\lambda}
{\sf s}_{{\sigma}^+(k\mapsto \lambda)}x)$$
$$={\sf s}_{\tau k}^{\lambda}({\sf d}_{\lambda, \sigma k}.{\sf s}_{{\sigma}^+(k\mapsto \lambda)}x)$$
$$= {\sf d}_{\tau k,  \sigma k}.{\sf s}_{\sigma^+(k\mapsto \tau k)}x.$$
By $F$ is a filter and $\tau k E\sigma k,$ we conclude that
$${\sf s}_{\sigma^+ }x\in F \text { iff }{\sf s}_{\sigma^+(k\mapsto \tau k)}x\in F.$$
The conclusion follows from the induction hypothesis.
We have proved that $f$ is well defined.
We now check that $f$ is a homomorphism, i.e. it 
preserves the operations.
Let $\sigma$ be a fixed element of $^{\omega}\alpha^{(Id)}$. 
Then since $F$ is maximal we have by 1.3 (iii) 
$$\bar{\sigma}\in f(x+y) \text { iff }
{\sf s}_{\sigma^+}(x+y)\in F\text { iff }{\sf s}_{\sigma^+}x+{\sf s}_{\sigma^+}y\in F$$
iff 
$${\sf s}_{\sigma^+} x \text { or } {\sf s}_{\sigma^+} y\in F \text 
{ iff }\bar{\sigma}\in f(x)\cup f(y).$$
We have proved that $f$ preserves the boolean join.
We now check complementation.
$$\bar{\sigma} \in f(-x) \text { iff }{\sf s}_{\sigma^+}(-x)\in F \text { iff }-{\sf s}_{\sigma^+x}\in F 
\text { iff }{\sf s}_{\sigma^+} x\notin F \text{ iff } \bar{\sigma}\notin f(x).$$
And diagonal elements. Let $k,l<\omega$ , then we have by 1.3 (ii)  
$$\sigma\in f{\sf d}_{kl}\text { iff } {\sf s}_{\sigma^+}{\sf d}_{kl}\in F \text { iff }
{\sf d}_{\sigma k, \sigma l}\in F
\text { iff }\sigma k E \sigma l \text { iff }$$
$$ \sigma k/E=\sigma l/E \text{ iff } 
\bar\sigma(k)=\bar\sigma(l)
\text{ iff } \bar{\sigma} \in d_{kl}.$$
Now cylindrifications. 

Let $k<\omega$. Let $\bar{\sigma}\in {\sf c}_kf(x)$. 
Then for some $\lambda<\alpha$
$$\bar{\sigma}(k \mapsto  \lambda/E)\in f(x)$$ 
hence 
$${\sf s}_{\sigma^+(k\mapsto \lambda)}x\in F$$ 
It follows from 1.3 (iii)  
and the inclusion $x\leq {\sf c}_kx$ that , 
$${\sf s}_{\sigma^+(k\mapsto \lambda)}{\sf c}_kx \in F$$
By 1.3 (vii) we have $s_{\sigma^+}{\sf c}_kx\in F.$
Thus ${\sf c}_kfx\subseteq f{\sf c}_kx.$

Now we prove the other inclusion.
Let $\bar\sigma\in f{\sf c}_kx$. Then 
$${\sf s}_{\sigma^+}{\sf c}_kx\in F.$$
Let $$\lambda\in \{\eta \in \alpha:  \sigma^{-1}\{\eta\}=\{\eta\}\}\smallsetminus \Delta x$$
Let $$\tau=\sigma\upharpoonright \alpha\smallsetminus \{k,\lambda\}\cup \{
(k,\lambda), (\lambda,k)\}.$$
Then by 1.3 (vii) (vi) we have
$${\sf c}_{\lambda}{\sf s}_{\tau} x={\sf s}_{\tau}{\sf c}_kx={\sf s}_{\sigma}{\sf c}_kx\in F.$$
By (3) there is some $u\notin \Delta(s_{\tau}x)$  
such that $${\sf s}_{u}^{\lambda}{\sf s}_{\tau}x\in F.$$
Thus by 1.3 (iv) and 1.3 (vi) we get
$${\sf s}_{\sigma^+(k\mapsto u)}x\in F$$
Hence $$\sigma(k\mapsto u)\in f(x).$$
Thus $\bar{\sigma}\in {\sf c}_kfx$.
Preserving substitutions follows from the definitions.
Recall that $W$ is  the weak space $^{\omega}M^{(q)}$
where $M=\alpha/E$ and $q(i)=i/E$ for every $i\in \omega.$

Let $f=f_1\cup f_2\upharpoonright X$. Then $f$ is a function since, be definition, $f_1$ and $f_2$ agree on $X_1\cap X_2$.
Moreover, by freeness,  $f$ extends to a homomorphism $g$ from $\A$ into 
$\langle \B(W), {\sf C}_i, {\sf S}_{\tau}, {\sf D}_{ij}\rangle_{i,j\in \omega,\tau\in [G]}$, that is $f\subseteq g$.
Since $a\in \Sg^{\Rd_{QPEA}\A}X_1$ then $g(a)=f_1(a)$ (by a straightforward induction, by noting that $g$ and $f_1$ are quasipolyadic 
homomorphisms).
Similarly, $g(-c) =f_2(-c).$ Now $Id\in f_1(a)\cap f_2(-c)$ hence 
$$g(a-c)=g(a)\cap g(-c)=f_1(a)\cap f_2(-c)\neq \emptyset.$$
This contradicts that $a\leq c$. The Theorem is proved. 
\end{demo}
\begin{theorem} Let $A,\B, \C\in V$, and assume that $\A,\B\subseteq  \C$ are cylindric subalgebras, then they have a superamalgam in $\CA$. 
superaamalgam over the cylindric reduct.
\end{theorem}
Let $X$ be a set of generators of $\A$ using only the cylindric operations, and $Y$ be that of $\B$.
Let $I=|X|$ and $J=|Y|$, and $|I\cup J|=\mu$, and let consider $\Fr_{\mu}V$.
Let $\Fr^{I}$ be the subalgebra of $\Fr$ generated by $\{x_i: i\in I\}$, using only the cylindric operations, and let 
$\Fr^J$ be the subalgebra generated by $\{y_i: i \in J\}$ using only the cylindric operations.
Then there exists  a surjective homomorphism from $\Rd_{ca}\Fr^{I}$ onto $\Rd_{ca}\A$
such that
$\xi i\mapsto a_i$ $(i\in I)$
and similarly 
a homomorphism from $\Rd_{ca}\Fr^{J}$ into $\Rd_{ca}\B$ such that
$\xi j\mapsto b_j$ $(j\in I).$
Therefore there exist ideals cylindric ideals $M$ and $N$ ideals of the cylindric algebras $\Fr^I$ and $\Fr^J$ respectively, and 
there exist isomorphisms 
$$m:\Fr^{I}/M\to \Rd_{ca}\A\text { and }n:\Fr^J/N\to \Rd_{ca}\B$$ such that
$$m(\xi i/M)=a_i\text { and }(\xi i/N)=b_i.$$
It is esay to check that $M\cap \Fr^{(I\cap J)}=N\cap \Fr^{(I\cap J)}.$
Now let $x\in \Ig(M\cup N)\cap \Fr^{I}$.
Then there exist $b\in M$ and $c\in N$ such that $x\leq b+c$. Thus $x-b\leq c$.
But $x-b\in \Fr^{(I)}$ and $c\in \Fr^{J}$, it follows that there exists an interpolant 
$d\in \Fr^{(I\cap J)}$ such that $x-b\leq d\leq c$. We have $d\in N$
therefore $d\in M$, and since $x\leq d+b$, therefore $x\in M$.
It follows that
$\Ig(M\cup N)\cap \Fr^{I}=M$
and similarly
$\Ig(M\cup N)\cap \Fr^{J}=N$.
In particular $P=\Ig(M\cup N)$ is a  proper cylindric ideal of $\Fr$.
Let $\D=\Fr/P$.
Let $k:\Fr^{I}/M\to \Fr/P$ be defined by $k(a/M)=a/P$
and $h:\Fr^{J}/M\to \Fr/P$ by $h(a/N)=a/P$. Then 
$k\circ m$ and $h\circ n$ are one to one and 
$k\circ m \circ f=h\circ n\circ g$. 
We now prove that $\Fr/P$ is actually a
superamalgam. i.e we prove that $K$ has the superamalgamation
property. Assume that $k\circ m(a)\leq h\circ n(b)$. There exists
$x\in \Fr^{I}$ such that $x/P=k(m(a))$ and $m(a)=x/M$. Also there
exists $z\in \Fr^{J}$ such that $z/P=h(n(b))$ and $n(b)=z/N$. Now
$x/P\leq z/P$ hence $x-z\in P$. Therefore  there is an $r\in M$ and
an $s\in N$ such that $x-r\leq z+s$. Now $x-r\in \Fr^I$ and $z+s\in
\Fr^J,$ it follows that there is an interpolant  $u\in \Fr^{(I\cap
J)}$ such that $x-r\leq u\leq z+s$. Let $t\in C$ such that $m\circ
f(t)=u/M$ and $n\circ g(t)=u/N.$ We have  $x/P\leq u/P\leq z/P$. Now
$m(f(t))=u/M\geq x/M=m(a).$ Thus $f(t)\geq a$. Similarly
$n(g(t))=u/N\leq z/N=n(b)$, hence $g(t)\leq b$. By total symmetry,
we are done.

\section{Relation algebras}

For relation algebras it is known that for any
$K$ such that $$\RRA\subseteq K\subseteq \RA,$$ 
$K$ does not have $AP$, a result of Comer.

For relation algebras, the following is a result of Monk:

\begin{theorem} $\RRA=S\Ra\CA_{\omega}$
\end{theorem}
$UNEP$ and $NS$ can be defined like the $\CA$ case, and the same characterization of bases hold.
Now we give a natural clas of $RA$'s that has $NS$ and $UNEP$.

We need some preparing. The pairing technique due to Alfred Tarski, and substantially generalized by Istvan N\;emeti, consists of  defining 
a pair of quasi-projections.
$p_0$ and $p_1$ 
so that in a model $\cal M$ say of  a certain sentence $\pi$, where $\pi$ is built out of these quasi-projections, $p_0$ and $p_1$
are functions and for any element $a,b\in {\cal M}$, there is a $c$ 
such that $p_0$ and $p_1$
map $c$ to $a$ and $b,$ respectively. 
We can think of $c$ as representing the ordered pair $(a,b)$ 
and $p_0$ and $p_1$ are the functions that project the ordered pair onto 
its first and second coordinates.  

Such a technique, ever since introduced by Tarski, to formalize, and indeed succesfully so, set theory, in the calculas of relations 
manifested itself in several re-incarnations in the literature some of which are quite subtle and 
sophisticated.
One is Simon's proof of the representability of quasi-relation algebras $\QRA$ 
(relation algebrs with quasi projections) using a neat embedding theorem for cylindric algebras \cite{Andras}. 
The proof consists of stimulating a neat embeding theorem
via the quasi-projections, in short it is actually a 
{\it a completeness proof}. The idea implemented  is that quasi-projections, on the one hand, generate extra dimensions, and on the other it has 
control over such a stretching. The latter property does not come across very much in Simon's proof, but below we will give an exact rigorous 
meaning to such property. This method can is used by Simon to
apply a Henkin completeness construction. We shall use Simon's technique to further show that $\QRA$ has the superamalgamation property; 
this is utterly unsurprising because Henkin constructions also prove interpolation theorems. This is the case, e.g. 
for first order logics and several of its non-trivial extensions arising from the process of algebraising first order logic, 
by dropping the condition of local finiteness reflecting the fact
that formulas contain only finitely many (free) variables. A striking example in this connection is the algebras studied by Sain and Sayed Ahmed 
\cite{Sain}, \cite{Sayed}. This last condition of loal finiteness is  unwarrented from the algebraic point of view, 
because it prevents an equational formalism of first order logic.

The view, of capturing extra dimensions, using also quai-projections comes along also very much so, 
in N\'emetis directed cylindric algebras (introduced as a $\CA$ counterpart of 
$\QRA$). In those, S\'agi defined quasi-projections also to achieve a completeness theorem for higher order logics.
The technique used is similar to Maddux's proof of representation of ${\QRA}$s, which further emphasizes the correlation.
We have alreardy made the notion of extra dimensions explicit. Its dual notion (in an exact categorial sense presented above)
that  of compressing dimensions, or taking neat reducts.

The definition of neat reducts in the standard definition adopted by Henkin, Monk and Tarski in their monograph, 
deals only with inisial segements, but it proves useful to widen the definition a little allowing 
arbitary substs of $\alpha$ not just initial segments. This is no more than a notational tactic.

This will enable us, using a deep result of Simon, to present an equivalence between algebras that are finite dimensional. We infer from our definition 
that such algebras, referred to in the literature as {\it Directed Cylindric Algebras}, actually 
belong to the polyadic paradigm, in this context the neat reduct functor establishes an equivalence between all dimensions.
To achieve this equivalent we use a transient category, namey, that of quasi-projective relation algebras.

\begin{definition} Let ${C}\in \CA_{\alpha}$ and $I\subseteq \alpha$, and let $\beta$ be the order type of $I$. Then
$$Nr_IC=\{x\in C: c_ix=x \textrm{ for all } i\in \alpha\sim I\}.$$
$$\Nr_{I}{\C}=(Nr_IC, +, \cdot ,-, 0,1, c_{\rho_i}, d_{\rho_i,\rho_j})_{i,j<\beta},$$
where $\beta$ is the unique order preserving one-to-one map from $\beta$ onto $I$, and all the operations 
are the restrictions of the corresponding operations on $C$. When $I=\{i_0,\ldots i_{k-1}\}$ 
we write $\Nr_{i_0,\ldots i_{k-1}}\C$. If $I$ is an initial segment of $\alpha$, $\beta$ say, we write $\Nr_{\beta}\C$.
\end{definition}
Similar to taking the $n$ neat reduct of a $\CA$, $\A$ in a higher dimension, is taking its $\Ra$ reduct, its relation algebra reduct.
This has unverse consisting of the $2$ dimensional elements of $\A$, and composition and converse are defined using one spare dimension.
A slight generalization, modulo a reshufflig of the indicies: 

\begin{definition}\label{RA} For $n\geq 3$, the relation algebra reduct of $\C\in \CA_n$ is the algebra
$$\Ra\C=(Nr_{n-2, n-1}C, +, \cdot,  1, ;, \breve{}, 1').$$ 
where $1'=d_{n-2,n-1}$, $\breve{x}=s_{n-1}^0s_{n-1}^{n-2}s_0^{n-1}x$ and $x;y=c_0(s_0^{n-1}x. s_0^{n-2}y)$. 
Here $s_i^j(x)=c_i(x\cdot d_{ij})$ when $i\neq q$ and $s_i^i(x)=x.$
\end{definition}
But what is not obvious at all is that an $\RA$ has a $\CA_n$ reduct for $n\geq 3$. 
But Simon showed that certain relations algebras do; namely the $\QRA$s.

\begin{definition} A relation algebra $\B$ is a $\QRA$ 
if there are elements $p,q$ in $\B$ satisfying the following equations:
\begin{enumarab}
\item $\breve{p};p\leq 1', q; q\leq 1;$
\item  $\breve{p};q=1.$
\end{enumarab}
\end{definition}
In this case we say that $\B$ is a $\QRA$  with quasi-projections $p$ and $q$. 
To construct cylindric algebras of higher dimensions 'sitting' in a $\QRA$, 
we need to define certain terms. seemingly rather complicated, their intuitive meaning 
is not so hard to grasp.
\begin{definition} Let $x\in\B\in \RA$, then $dom(x)=1';(x;\breve{x})$ and $ran(x)=1';(\breve{x}; x)$, $x^0=1'$, $x^{n+1}=x^n;x$. $x$ 
is a functional element if $x;\breve{x}\leq 1'$.
\end{definition}
Given a $\QRA$, which we denote by $\bold Q$, we have quasi-projections $p$ and $q$ as mentioned above. 
Next we define certain terms in ${\bf Q}$, cf. \cite{Andras}:


$$\epsilon^{n}=dom q^{n-1},$$
$$\pi_i^n=\epsilon^{n};q^i;p,  i<n-1, \pi_{n-1}^{(n)}=q^{n-1},$$
$$ \xi^{(n)}=\pi_i^{(n)}; \pi_i^{(n)},$$
$$ t_i^{(n)}=\prod_{i\neq j<n}\xi_j^{(n)}, t^{(n)}=\prod_{j<n}\xi_j^{(n)},$$
$$ c_i^{(n)}x=x;t_i^{(n)},$$
$$ d_{ij}^{(n)}=1;(\pi_i^{(n)}.\pi_j^{(n)}),$$
$$ 1^{(n)}=1;\epsilon^{(n)}.$$
and let
$$\B_n=(B_n, +, \cdot, -, 0,1^{(n)}, c_i^{(n)}, d_{ij}^{(n)})_{i,j<n},$$
where $B_n=\{x\in B: x=1;x; t^{(n)}\}.$
The intuitive meaning of those terms is explained in \cite{Andras}, right after their definition on p. 271.

\begin{theorem} Let $n>1$
\begin{enumerate} 
\item Then ${\B}_n$ is closed under the operations.
\item ${\B}_n$ is a $\CA_n$.
\end{enumerate}
\end{theorem}
\begin{proof} (1) is  proved in \cite{Andras} lemma 3.4 p.273-275 where the terms are definable in a $\QRA$. 
That it is a $\CA_n$ can be proved as \cite{Andras}  theorem 3.9.
\end{proof} 

\begin{definition} Consider the following terms.
$$suc (x)=1; (\breve{p}; x; \breve{q})$$
and
$$pred(x)=\breve{p}; ranx; q.$$ 
\end{definition}
It is proved in \cite{Andras} that $\B_n$ neatly embeds into $\B_{n+1}$ via $succ$. The successor function thus codes 
extra dimensions. The thing to observe here is that  we will see that $pred$; its inverse; 
guarantees a condition of commutativity of two operations: forming neat reducts and forming subalgebras;
it does not make a difference which operation we implement first, as long as we implement both one after the other.
So the function $succ$ {\it captures the extra dimensions added.}. From the point of view of {\it definability} it says 
that terms definable in extra dimensions add nothing, they are already term definable.
And this indeed is a definability condition, that will eventually lead to stong interpolation property we wnat.

\begin{theorem}\label{neat} Let $n\geq 3$. Then $succ: {\B}_n\to \{a\in {\B}_{n+1}: c_0a=a\}$ 
is an isomorphism into a generalized neat reduct of ${\B}_{n+1}$.
Strengthening the condition of surjectivity,  for all $X\subseteq \B_n$, $n\geq 3$, we have (*)
$$succ(\Sg^{\B_n}X)\cong \Nr_{1,2,\ldots, n}\Sg^{\B_{n+1}}succ(X).$$
\end{theorem}

\begin{proof} The operations are respected by \cite{Andras} theorem 5.1. 
The last condition follows  because of the presence of the 
functional element $pred$, since we have $suc(pred x)=x$ and $pred(sucx)=x$, when $c_0x=x$, \cite{Andras} 
lemmas 4.6-4.10. 
\end{proof}
\begin{theorem}
Let $n\geq 3$. Let ${\C}_n$ be the algebra obtained from ${\B}_n$ by reshuffling the indices as follows; 
set $c_0^{{\C}_n}=c_n^{{\B}_n}$ and $c_n^{{\C}_n}=c_0^{{\cal B}_n}$. Then ${\C}_n$ is a cylindric algebra,
and $suc: {\C}_n\to \Nr_n{\C}_{n+1}$ is an isomorphism for all $n$. 
Furthermore, for all $X\subseteq \C_n$ we have
$$suc(\Sg^{\C_n}X)\cong \Nr_n\Sg^{\C_{n+1}}suc(X).$$ 
\end{theorem}

\begin{proof} immediate from \ref{neat}
\end{proof}  
\begin{theorem} Let ${\C}_n$ be as above. Then $succ^{m}:{\C_n}\to \Nr_n\C_m$ is an isomophism, such that 
for all $X\subseteq A$, we have
$$suc^{m}(\Sg^{\C_n}X)=\Nr_n\Sg^{\C_m}suc^{n-1}(X).$$
\end{theorem} 
\begin{proof} By induction on $n$.
\end{proof}
Now we want to neatly embed our $\QRA$ in $\omega$ extra dimensions. At the same we do not want to lose, our control over the streching;
we still need the commutativing of taking, now $\Ra$  reducts with forming subalgebras; we call this property the $\Ra S$ property.
To construct the big $\omega$ dimensional algebra, we use a standard ultraproduct construction.
So here we go.
For $n\geq 3$, let  ${\C}_n^+$ be an algebra obtained by adding $c_i$ and $d_{ij}$'s for $\omega>i,j\geq n$ arbitrarity and with 
$\Rd_n^+\C_{n^+}={\B}_n$. Let ${\C}=\prod_{n\geq 3} {\C}_n^+/G$, where $G$ is a non-principal ultrafilter
on $\omega$. 
In our next theorem, we show that the algebra $\A$ can be neatly embedded in a locally finite algebra $\omega$ dimensional algebra
and we retain our $\Ra S$ property. 

\begin{theorem} Let $$i: {\A}\to \Ra\C$$
be defined by
$$x\mapsto (x,  suc(x),\ldots suc^{n-1}(x),\dots n\geq 3, x\in B_n)/G.$$ 
Then $i$ is an embedding ,
and for any $X\subseteq A$, we have 
$$i(\Sg^{\A}X)=\Ra\Sg^{\C}i(X).$$
\end{theorem}
\begin{proof} The idea is that if this does not happen, then it will not happen in a fnite reduct, and this impossible \cite{Sayed}.

\end{proof}

Note that Simon's theorem, actually says that in every $\QRA$, there sits an $\RCA_3$.

\begin{theorem} Let $\bold Q\in {\RA}$. Then for all $n\geq 4$, there exists a unique 
$\A\in S\Nr_3\CA_n$ such that $\bold Q=\Ra\A$, 
such that for all $X\subseteq A$, $\Sg^{\bf Q}X=\Ra\Sg^{\A}X.$
\end{theorem}
\begin{proof} This follows from the previous theorem together with $\Ra S$ property.
\end{proof}

\begin{corollary} Assume that $Q=\Ra\A\cong \Ra\B$ then this lifts to an isomorphism from $\A$ to $\B$.
\end{corollary}
The previous theorem says that $\Ra$ as a functor establishes an equivalence between ${\QRA}$ 
and a reflective subcategory of $\Lf_{\omega}.$
We say that $\A$ is the $\omega$ dilation of ${\bf Q}$.
Now we are ready for:

\begin{theorem} $\QRA$ has $SUPAP$.
\end{theorem}
\begin{proof}  We form the unique dilatons of the given algebras required to be superamalgamated. 
These are locally finite so we can find a superamalgam $\D$. 
Then $\Ra\D$ will be required superamalgam; it contains quasiprojections because the base algebras 
does.
Let $\A,\B\in \QRA$. Let $f:\C\to \A$ and $g:\C\to \B$ be injective homomorphisms .
Then there exist $\A^+, \B^+, \C^+\in \CA_{\alpha+\omega}$, $e_A:\A\to \Ra{\alpha}\A^+$ 
$e_B:\B\to  \Ra\B^+$ and $e_C:\C\to \Ra\C^+$.
We can assume, without loss,  that $\Sg^{\A^+}e_A(A)=\A^+$ and similarly for $\B^+$ and $\C^+$.
Let $f(C)^+=\Sg^{\A^+}e_A(f(C))$ and $g(C)^+=\Sg^{\B^+}e_B(g(C)).$
Since $\C$ has $UNEP$, there exist $\bar{f}:\C^+\to f(C)^+$ and $\bar{g}:\C^+\to g(C)^+$ such that 
$(e_A\upharpoonright f(C))\circ f=\bar{f}\circ e_C$ and $(e_B\upharpoonright g(C))\circ g=\bar{g}\circ e_C$. Both $\bar{f}$ and $\bar{g}$ are 
monomorphisms.
Now $\Lf_{\omega}$ has $SUPAP$, hence there is a $\D^+$ in $K$ and $k:\A^+\to \D^+$ and $h:\B^+\to \D^+$ such that
$k\circ \bar{f}=h\circ \bar{g}$. $k$ and $h$ are also monomorphisms. Then $k\circ e_A:\A\to \Ra\D^+$ and
$h\circ e_B:\B\to \Ra\D^+$ are one to one and
$k\circ e_A \circ f=h\circ e_B\circ g$.
Let $\D=\Ra\D^+$. Then we obtained $\D\in \QRA$ 
and $m:\A\to \D$ $n:\B\to \D$
such that $m\circ f=n\circ g$.
Here $m=k\circ e_A$ and $n=h\circ e_B$. 
Denote $k$ by $m^+$ and $h$ by $n^+$.
Now suppose that $\C$ has $NS$. We further want to show that if $m(a) \leq n(b)$, 
for $a\in A$ and $b\in B$, then there exists $t \in C$ 
such that $ a \leq f(t)$ and $g(t) \leq b$.
So let $a$ and $b$ be as indicated. 
We have  $(m^+ \circ e_A)(a) \leq (n^+ \circ e_B)(b),$ so
$m^+ ( e_A(a)) \leq n^+ ( e_B(b)).$
Since $K$ has $SUPAP$, there exist $z \in C^+$ such that $e_A(a) \leq \bar{f}(z)$ and
$\bar{g}(z) \leq e_B(b)$.
Let $\Gamma = \Delta z \sim \alpha$ and $z' =
{\sf c}_{(\Gamma)}z$. (Note that $\Gamma$ is finite.) So, we obtain that 
$e_A({\sf c}_{(\Gamma)}a) \leq \bar{f}({\sf c}_{(\Gamma)}z)~~ \textrm{and} ~~ \bar{g}({\sf c}_{(\Gamma)}z) \leq
e_B({\sf c}_{(\Gamma)}b).$ It follows that $e_A(a) \leq \bar{f}(z')~~\textrm{and} ~~ \bar{g}(z') \leq e_B(b).$ Now by hypothesis
$$z' \in \Ra\C^+ = \Sg^{\Ra\C^+} (e_C(C)) = e_C(C).$$ 
So, there exists $t \in C$ with $ z' = e_C(t)$. Then we get
$e_A(a) \leq \bar{f}(e_C(t))$ and $\bar{g}(e_C(t)) \leq e_B(b).$ It follows that $e_A(a) \leq (e_A \circ f)(t)$ and 
$(e_B \circ g)(t) \leq
e_B(b).$ Hence, $ a \leq f(t)$ and $g(t) \leq b.$
We are done.
\end{proof}

\end{document}